\documentclass[letter, 11pt]{article}

\usepackage{natbib}
\usepackage{amssymb,amsmath,xcolor,graphicx,xspace,colortbl,rotating} %
\usepackage[raggedrightboxes]{ragged2e} 
\usepackage{textcomp}
\usepackage{amssymb,amsmath,xcolor,graphicx,xspace,colortbl,rotating}  
\usepackage{appendix}
\usepackage{mathrsfs}
\usepackage{bm}  
\usepackage{boxedminipage}  
\usepackage{color}  
\usepackage{endnotes}  
\usepackage[margin=1in]{geometry}
\usepackage[raggedrightboxes]{ragged2e}  
\usepackage[doublespacing]{setspace}  
\usepackage{tabulary}  
\usepackage{textcomp}  
\usepackage{varioref}  
\usepackage{wrapfig}  
\graphicspath{{../graphics/}{../tcache/}{../gcache/}}
\DeclareGraphicsExtensions{.pdf,.eps,.ps,.png,.jpg,.jpeg}

\newtheorem {theorem}{Theorem}

\newtheorem {corollary}{Corollary}

\newtheorem {definition}{Definition}

\newtheorem {lemma}{Lemma}

\newtheorem {proposition}{Proposition}
\newtheorem {remark}{Remark}

\newenvironment {proof}[1][Proof]{\noindent \textbf {#1.} }{\ \rule {0.5em}{0.5em}}

\begin{document}
\title{New Jensen-type inequalities and their applications}
\author{Bar Light\protect\footnote{  Graduate School of Business,
Stanford University, Stanford, CA 94305, USA. e-mail: \textsf{barl@stanford.edu}\ }
~ ~
}
\maketitle

\thispagestyle{empty}

\noindent \noindent \textsc{Abstract}:
\begin{quote}
 Convex analysis is fundamental to proving inequalities that have a wide variety of applications in economics and mathematics. In this paper we provide Jensen-type inequalities for functions that are, intuitively, ``very" convex. These inequalities are simple to apply and can be used to generalize and extend previous results or to derive new results. We apply our inequalities to quantify the notion ``more risk averse" provided in \cite{pratt1978risk}. We also apply our results in other  applications from different fields, including risk measures, Poisson approximation,   moment generating functions, log-likelihood functions, and Hermite-Hadamard type inequalities.
\end{quote}

\noindent {\small Keywords: Convexity, $(p,a,b)$-convex  functions, risk aversion, risk measures, moment generating functions, log-likelihood functions. }  \\\relax
\smallskip \noindent \emph{} 

\newpage 

\section{Introduction}
  
Let $f: \mathbb{R} \rightarrow \mathbb{R}$ be a convex function. Suppose that $a$ minimizes $f$, i.e., $f(x) \geq f(a)$ for all $x \in \mathbb{R}$ and that $f(a)=0$. Then, intuitively, the antiderivative  of the function $f$ on $[a,\infty)$ given by $F(x) : = \int _{a}^{x} f(x) dx$ is ``more" convex than $f$ on $[a,\infty)$. For example, if $f(x)=\vert x \vert $ then $F(x) = x^2/2$ is ``more" convex than $f$ on $[0,\infty)$. Similarly, the antiderivative  of $F$ is intuitively more convex than $F$, and so on. In this paper we provide Jensen-type inequalities for $F$ and its antiderivatives. These inequalities are tighter than the standard inequalities that hold for convex functions. Importantly, these inequalities are simple and can be used in various applications. 

We demonstrate the usefulness of our results in a variety of applications from different fields that are of independent interest. In our first application, we generalize a well-known result by \cite{pratt1978risk} that provides conditions that imply that one expected utility decision maker is more risk averse than another. We provide conditions that quantify the relation `more risk averse'. In our second application, we provide risk measures that are based on utility functions. In our third application we provide a novel Poisson approximation in terms of the Wasserstein distance. In our fourth application, we provide bounds on the moment generating function of a random variable that involve the random variable's first $p$ moments. We also provide a bound on the expected value of a random variable that generalizes the AM-GM inequality. In our fifth application, we provide lower bounds for the log-likelihood function in a standard statistical setting with hidden variables where directly maximizing the log-likelihood function is usually intractable. In the sixth application, we derive novel Hermite-Hadamard type inequalities.

The rest of the paper is organized as follows. Section 2 introduces the $(p,a,b)$-convex functions and provides inequalities that involve these functions.  In Section 3 we use the results from Section 2 for various applications. In Section 4 we provide a summary. 

\section{The family of $(p,a,b)$-convex functions} \label{Sec: convex functions}

Throughout the paper we consider a fixed probability space $\left (\Omega ,\mathcal{F} ,\mathbb{P}\right )$. A random variable $X$ is a measurable real-valued function from $\Omega $ to $\mathbb{R}$. We denote the expectation of a random variable on the probability space $\left (\Omega  ,\mathcal{F} ,\mathbb{P}\right )$ by $\mathbb{E}$. For $1 \leq p \leq \infty $ let $L^{p} : =L^{p}\left (\Omega  ,\mathcal{F} ,\mathbb{P}\right )$ be the space of all random variables $X :\Omega  \rightarrow \mathbb{R}$ such that $\left \Vert X\right \Vert _{p}$ is finite, where $\left \Vert X\right \Vert _{p} =\left (\int _{\Omega }\left \vert X(\omega )\right \vert ^{p} \mathbb{P}(d \omega) \right  )^{1/p}$ for $1 \leq p <\infty $ and $\left \Vert X\right \Vert _{\infty} =\operatorname{ess\,sup} \left \vert X(\omega )\right \vert $. We say $X$ is a random variable on $[a,b]$ for some $a<b$ if $\mathbb{P}(X \in [a,b]) = 1$.


Let $C^{p}([a,b])$ be the set of all $p$ times continuously differentiable functions $f :[a ,b] \rightarrow \mathbb{R}$. For $k\geq 1$, we denote by $f^{(k)}$ the $k$th derivative of a function $f$ and for $k=0$ we define $f^{(0)}:=f$. As usual, the derivatives at the extreme points $f^{(k)}(a)$ and $f^{(k)}(b)$ are defined by taking the left-side and right-side limits, respectively. 
 
 For a non-negative integer and real numbers $a<b$ we define the following set of functions: 
\begin{equation*}\mathfrak{I}(p ,a ,b) : =\{f \in C^{p}([a,b]): \: f^{(p)} \text{ is convex and increasing, } \: f^{(k)}(a) =0 \; \forall k =1 ,\ldots  ,p \}.
\end{equation*}
For an integer $p\geq 1$ we say that a function $f$ is a $(p,a,b)$-convex function if $f \in \mathfrak{I}(p ,a ,b)$. For $p=0$ we say that a function $f$ is a $(0,a,b)$-convex function if $f$ is convex on $[a,b]$. For an integer $p \geq 1$, a function $f$ is a $(p,a,b)$-convex function if the $p$th derivative of $f$ is a convex and increasing function and $f^{(k)}(a)=0$ for all $k=1,\ldots,p$. For every positive integer $p$ the set $\mathfrak{I}(p ,a ,b)$ is a subset of the set of convex and increasing functions. 
The class of functions $\mathfrak{I}(p ,- \infty ,\infty)$ is widely studied in the literature (for an early reference see \cite{williamson1955multiply}) and  plays an important rule in deriving some concentration inequalities (see \cite{pinelis1999fractional}) and stochastic orders (see \cite{fishburn1980stochastic}). The class of $(1,a,b)$-convex functions is used in \cite{light2019family} to study stochastic orders. 

The antiderivative of a $(p-1,a,b)$-convex function $g$ such that $g(a)=0$ is a $(p,a,b)$-convex function. More generally, suppose that $g$ is a $(p-1,a,b)$-convex function. It is easy to see that the function $f(x) := \int _{a}^{x} (g(z)-g(a))dz$ is a $(p,a,b)$-convex function. In particular, the functions that belong to the set $\mathfrak{I}(1 ,a ,b)$ can be identified as the integrals of convex functions. 

Functions that are $(p,a,b)$-convex arise naturally in many settings. The next simple observation shows that we can construct a $(p,a,b)$-convex from the Taylor series of a convex function. We will use this observation in Section 3. 
For a function $f:[0,b] \rightarrow \mathbb{R}$ that belongs to $C^{p}[0,b]$, the remainder of the Taylor series of order $p$ at the point $0$ is a $(p,0,b)$-convex function whenever the function $f^{(p)}$ is convex and increasing.

\begin{lemma} \label{lemma: Taylor}
Let $p \geq 1$ be an integer and let $f \in C^{p}[0,b]$,  $b>0$. If $f^{(p)}$ is convex and increasing on $[0,b]$ then $R_{f,p}$ is a $(p,0,b)$-convex function where $R_{f,p}(x) := f(x)  - \sum _{j=0} ^{p} f^{(j)}(0)x^{j} / j!$ is the  remainder of the Taylor series of order $p$ at the point $0$. 
\end{lemma}

\begin{proof}
Differentiating yields $R_{f,p}^{(k)}(0)= f^{(k)}(0) - f^{(k)}(0) = 0$ for all $k=1,\ldots,p$.

In addition, $R_{f,p}^{(p)}(x)= f^{(p)}(x) - f^{(p)}(0)$ is convex and increasing because $f^{(p)}$ is convex and increasing. We conclude that $R_{f,p}$ is a $(p,0,b)$-convex function.  
\end{proof}

The next theorem provides a version of Jensen's inequality for $(p,a,b)$-convex functions. The proof is deferred to the Appendix. 

\begin{theorem} \label{Thm: Jensen}
 Let $X$ be a random variable on $[a ,b]$ for some $a<b$. Let $p \geq 1$ be an integer.  

(i) For every $(p,a,b)$-convex function $f$ we have
\begin{equation}\mathbb{E}f(X) \geq f\left (a +\left (\mathbb{E}(X -a)^{p+1}\right )^{1/(p+1)}\right ) = f \left (a + \Vert X-a \Vert _{p+1} \right ). \label{Ineq: Jensen}
\end{equation}

(ii) Let $f$ be a $(p,a,b)$-convex function and assume that $f(a)=0$. Let $g(x) := f(x)/(x-a)^{p+1}$. Then $g$ is increasing on $(a,b)$.

\end{theorem}

\begin{remark}
(i) From a standard truncation argument Theorem \ref{Thm: Jensen} also holds for a random variable $X$ on $[a, \infty)$ such that $X \in L^{p+1}$. 

(ii) Let $f$ be a $(p,a,b)$-convex function. For every integer $p \geq 1$ and every random variable $X \in L^{p+1}$ on $[a, \infty)$, we have $a +\left \Vert X -a\right \Vert _{p+1} \geq a +\mathbb{E}(X -a) =\mathbb{E}(X)$. 

Combining the last inequality with the fact that $f$ is an increasing function (note that a $(p,a,b)$-convex function is increasing) yields 
\begin{equation*}f\left (a +\left \Vert X -a\right \Vert _{p+1}\right ) \geq f\left (\mathbb{E}(X)\right ).
\end{equation*} 
Thus, for all $p \geq 1$, Theorem \ref{Thm: Jensen} is tighter than Jensen's inequality. That is, for functions that belong to the set $\mathfrak{I}(p ,a ,b)$ we have a tighter lower bound on $\mathbb{E}f(X)$ than the standard lower bound of $f\left (\mathbb{E}(X)\right )$. 
\end{remark}

Using Theorem \ref{Thm: Jensen},  we can also  derive upper bounds on $\mathbb{E}f(X)$ for a bounded random variable $X$ and for a function $f \in \mathfrak{I}(p ,a ,b)$ that are tighter than the standard bound derived from Jensen's inequality. These bounds depend on the first $p$ moments of the random variable $X$. The proof is deferred to the Appendix.  

\begin{corollary}\label{Corollary: jensen upper bound}
 Let $X$ be a random variable on $[a ,b]$ for some $a<b$. Let $p \geq 1$ be an integer.  
Then, for every $(p,a,b)$-convex function $f$ we have
\begin{equation*}
   \left(1-\frac{\mathbb{E} (X-a) ^{p+1}}{(b-a)^{p+1}} \right)f(a) + \frac{\mathbb{E} (X-a) ^{p+1}}{(b-a)^{p+1}} f(b) \geq \mathbb{E} f(X)
\end{equation*} 
\end{corollary}

Theorem \ref{Thm: Jensen} and Corollary \ref{Corollary: jensen upper bound} hold for convex and increasing functions. We can prove similar results for convex and decreasing functions.

 For a non-negative integer and real numbers $a<b$ we define the following set of functions: 
\begin{equation*}\mathfrak{D}(p ,a ,b) : =\{f \in C^{p+2}([a,b]): \: (-1)^{k} f^{(k)} \geq 0 \text{ } \forall k =1 ,\ldots  ,p+2,  \text{ }  f^{(k)}(b) =0 \; \forall k =1 ,\ldots  ,p \}.
\end{equation*}
The proof of the following Proposition is similar to the proof of Theorem \ref{Thm: Jensen} and is therefore omitted. 
\begin{proposition} \label{Prop: JEnsen}
Let $X$ be a random variable on $[a,b]$ for some $a<b$. Suppose that $f \in \mathfrak{D}(p ,a ,b)$ for some integer $p \geq 1$. Then 
\begin{equation}\mathbb{E}f(X) \geq  f \left (b - \Vert b - X \Vert _{p+1} \right ). \label{Ineq: JensenDECREASING}
\end{equation}
\end{proposition}

\section{Applications}

\subsection{Risk aversion} \label{sec: risk aversion}
 
 Consider a setting in which a decision maker (DM) faces a lottery that is represented by some random variable $X$ on $[0,\infty)$.\footnote{Our results can be generalized to any random variable $Y$ on $[a,-\infty)$ that is bounded below by considering the random variable $X:=Y-a$ on $[0,\infty)$.} A realization of $X$, say $x$, represents a loss of $x$ dollars. A loss  function $l:[0,\infty) \rightarrow [0,\infty)$ is a strictly convex and strictly increasing function. For a DM with a loss function $l$ the expected loss from  a risky lottery $X$ is given by $\mathbb{E}l(X)$. The convexity of the loss function  represents the decision maker's risk aversion. Let $f$ and $l$ be two loss functions on $[0,\infty)$. A question of interest is whether a DM with a loss function $l$ is more risk averse than a DM with a loss function $f$. A standard definition \citep{pratt1978risk} states that $l$ exhibits more risk aversion than $f$ if,  for every number $c$ and every lottery $X$, whenever $l$ prefers the lottery $X$ to some sure amount $c$ then $f$ also prefers the lottery $X$ to $c$. In this section we are interested in extending this definition to formalize the following: To what degree does the loss function $l$ exhibit more risk aversion than the loss function $f$? In the spirit of \cite{pratt1978risk} we introduce the following definition: 
 
 \begin{definition}
 Let $l$ and $f$ be two loss functions and let $p \geq 1$ be an integer.  We say that $l$ is $p$-more risk averse than $f$ if $\Vert f(X)  \Vert _{p} \leq f(c)$ whenever $\mathbb{E}l(X) \leq l(c)$ for every number $c \geq 0$ and every random variable $X$ on $[0,\infty)$.
 \end{definition}
 
 Note that for $p=1$ the binary relation $1$-more risk averse reduces to the standard more risk averse binary relation that we mentioned above \citep{pratt1978risk}. For $p \geq 2$, because $\Vert f(X)  \Vert _{p} \geq \Vert f(X)  \Vert _{1} $, we require a stronger condition. This stronger condition provides a natural way to quantify the relation ``more risk averse". Hence, the relation $p$-more risk averse captures the degree to which $l$ is more risk averse than $f$. When $p$ is higher, the degree to which $l$ is more risk averse than $f$ is higher. The next Theorem shows that a simple characterization of $p$-more risk aversion can be provided in terms of the $(p,a,b)$-convex functions. Note that \cite{pratt1978risk} proves Theorem \ref{Theorem pratt} for the case that $p=1$. As usual, for two functions $f$ and $g$ we write $(f \circ g)(x) :=f(g(x)) $.

 \begin{theorem} \label{Theorem pratt}
Let $l$ and $f$ be two loss functions and let $p \geq 1$ be an integer. The function $l \circ f^{-1}$ is $(p-1,0,\infty)$-convex if and only if $l$ is $p$-more risk averse than $f$.  
 \end{theorem}
\begin{proof}
Let $p \geq 1$ be an integer. Suppose that $l \circ f^{-1}$ is $(p-1,0,\infty)$-convex. Let $X$ be a random variable on $[0,\infty)$ and let $c \geq 0$ be such that $\mathbb{E}l(X) \leq l(c)$. Note that $$\mathbb{E}l(X) \leq l(c) \Leftrightarrow \mathbb{E}( l \circ f^{-1} ) (f(X)) \leq (l \circ f^{-1}) (f(c)).$$ 
Theorem \ref{Thm: Jensen} applied for the random variable $Z=f(X)$ on $[0,\infty)$ implies that 
$$( l \circ f^{-1} ) ( \Vert f(X) \Vert _{p} ) \leq \mathbb{E}( l \circ f^{-1})(f(X)). $$ We conclude that $ ( l \circ f^{-1} ) ( \Vert f(X) \Vert _{p} ) \leq  (l \circ f^{-1}) (f(c))$. 
Because $l$ and $f$ are strictly increasing we have $\Vert f(X) \Vert _{p} \leq f(c)$. Thus,  $l$ is $p$-more risk averse than $f$. 

Now assume that $l$ is $p$-more risk averse than $f$.  Assume in contradiction that $l \circ f^{-1} $ is not $(p-1,0,\infty)$-convex. Then the function $k(z) =(l \circ f^{-1})(z^{1/p})$ is not convex  (see the proof of Theorem \ref{Thm: Jensen}). Thus, there exists $z_{1},z_{2} \geq 0$ and $\lambda \in (0,1)$ such that 
$$ ( l \circ f^{-1} ) \left ( \left ( \lambda z_{1} + (1-\lambda)z_{2} \right )^{1/p} \right ) > \lambda (l \circ f^{-1}) (z_{1}^{1/p}) + (1- \lambda) ( l \circ f^{-1} ) (z_{2}^{1/p} ) $$
Let $x_{1},x_{2},c$ be such that $f^{-1} (z_{i}^{1/p}) = x_{i}$ for $i=1,2$ and $l (x_{1}) + (1- \lambda)l (x_{2}) = l(c)$. We have 
\begin{align*}
     ( l \circ f^{-1} ) \left ( \left ( \lambda z_{1} + (1-\lambda)z_{2} \right )^{1/p} \right ) & > \lambda (l \circ f^{-1}) (f(x_{1})) + (1- \lambda) ( l \circ f^{-1} ) (f(x_{2})) \\
     & = \lambda l (x_{1}) + (1- \lambda)l (x_{2}) = l(c).
     \end{align*}
The fact that $l$ is strictly increasing implies 
$$ f^{-1}  \left ( \left ( \lambda z_{1} + (1-\lambda)z_{2} \right )^{1/p} \right ) > c \Leftrightarrow   \left ( \lambda f(x_{1})^{p} + (1-\lambda)f(x_{2})^{p} \right )^{1/p} > f(c). $$
Thus, $\Vert f(X) \Vert _{p} > f(c)$ and $\mathbb{E}l(X) = l(c)$ for the random variable $X$ that yields $x_{1}$ with probability $\lambda$ and $x_{2}$ with probability $1 -\lambda$ which is a contradiction to the statement  $l$ is $p$-more risk averse than $f$.  
 \end{proof}

\subsection{Risk measures}  
 
 Consider the setting of Section \ref{sec: risk aversion}, where a decision maker (DM) faces a possible future loss that is represented by a random variable $X$ on $[0,\infty)$. The DM wants to measure the risk of the random variable $X$. A standard approach to measuring the risk of $X$ is to assume that the DM has some loss function $l$ but the DM does not know the law of $X$. The DM considers some set $Q$ that consists of possible laws for $X$. The risk of the random variable $X$ is measured by $\sup _{q \in Q} \mathbb{E}_{q} l(X)$. That is, the risk of $X$ is given by the worst-case expected loss given that the law of $X$ belongs to $Q$. This approach is fundamental in the theory of risk measures (see  \cite{artzner1999coherent}). A decision-theoretic axiomatization of this approach is provided in \cite{gilboa1989maxmin}.
 
 Alternatively, a different approach to measuring the risk of $X$ assumes that the DM knows the law of $X$ but does not know the loss function $l$. The DM considers some set $\mathcal{L}$ that consists of possible ``very" risk averse  loss functions (recall from Section \ref{sec: risk aversion} that a loss function $l:[0,\infty) \rightarrow [0,\infty)$ is a strictly increasing and strictly convex function). For every loss function $l \in \mathcal{L}$ the DM computes the certainty equivalent of $X$, i.e., $l^{-1} (\mathbb{E}l(X))$, and the risk of the random variable $X$ is measured by $R_{\mathcal{L}}(X):=\inf _{l \in \mathcal{L}} l^{-1} (\mathbb{E}l(X))$. Intuitively, $R_{\mathcal{L}}(X)$ equals the highest number $c$ such that for every loss function in the set $\mathcal{L}$ the expected loss is higher than or equal to $c$. That is, $c$ is the highest number such that every agent in the set of very risk averse agents $\mathcal{L}$ prefers $c$ to the lottery.\footnote{A similar idea and a decision-theoretic axiomatization are provided in \cite{cerreia2015cautious}.} 
 To practically use this approach one needs to characterize the risk measure  $ R_{\mathcal{L}}(X) $ for plausible sets of loss functions. Using Theorem \ref{Thm: Jensen} we characterize $ R_{\mathcal{L}}(X) $ for a set of loss functions that we now introduce.

For every integer $p \geq 1$ we consider the following sets of functions 
$$ \mathcal{L}_{p} := \{ l \in C^{p+2}[0,\infty): l^{(2)}(x)x \geq pl^{(1)}(x) \text { }  \forall x \geq 0, l^{(k)}(x) > 0 \text { } \forall  x > 0 \text {  }   \forall k=1,\ldots,p+2 \}. $$

Suppose that $l \in \mathcal{L}_{1}$. The requirements  $l^{(1)} > 0$ and $l^{(2)} > 0$ ensure that $l$ is a loss function. $l^{(3)} \geq 0$ means that the DM exhibits downside risk aversion which is a natural property in our setting (see \cite{menezes1980increasing}). The condition $l^{(2)}(x)x \geq l^{(1)}(x)$ is a curvature condition that is widely analyzed and used in the literature. Thus, the set $\mathcal{L}_{1}$ is a plausible set of ``very" risk averse loss functions. For a higher $p$, the loss functions in $\mathcal{L}_{p}$ exhibit higher-order risk aversion and a higher lower bound on the Arrow-Pratt measure of relative risk aversion. The next theorem characterizes $R_{\mathcal{L}_{p}}$ for every integer $p \geq 1$. Thus, the loss functions in $\mathcal{L}_{p}$ are more risk averse when $p$ is higher. Using Theorem \ref{Thm: Jensen} we show that $R_{\mathcal{L}_{p}}(X) = \Vert X \Vert _{p+1}$ for every integer $p \geq 1$. 

\begin{theorem} \label{Theorem: risk measures}
Let $p \geq 1$ be an integer. Suppose that $X \in L^{p+1}$. Then $R_{\mathcal{L}_{p}}(X) = \Vert X \Vert _{p+1}$.
\end{theorem}

\begin{proof}
Let $p \geq 1$ be an integer. Suppose that $l \in \mathcal{L}_{p}$. 
Define $z(x) =l^{(2)}(x)x -pl^{(1)}(x)$ on $[0 ,\infty)$. Because  $l \in \mathcal{L}_{p}$ we have $z(x) \geq 0$ for all $x \geq 0$. In particular $z(0) = -l^{(1)}(0) p \geq 0$ which implies that $l^{(1)}(0)=0$. Thus, $z(0)=0$.  Note that  \begin{equation*}z^{(k)}(x) =l^{(k +2)}(x)x -l^{(k +1)}(x)(p -k)
\end{equation*}
for all $k=0,\ldots,p$. Assume in contradiction that there exists some $j = 2,\ldots,p$ such that $l^{(j)}(0) > 0$. Then there exists an integer $2 \leq m \leq p$ that satisfies  $l^{(m)}(0) > 0$ and $l^{(k)}(0) = 0$ for all  $k=1,\ldots,m-1$. We have  $z^{(k)}(0)=0$ for all $k=0,\ldots,m-2$ and $z^{(m-1)}(0) < 0$ which is a contradiction to the fact that $z^{(0)}(x):=z(x) \geq 0$ for all $x \geq 0$. We conclude that $l ^{(k)} (0) = 0$ for all $k=1,\ldots,p$. Because $l \in \mathcal{L}_{p}$ the function $l^{(p)}$ is convex and increasing. Thus, $l$ is a $(p,0,\infty)$-convex function. 

Using Theorem \ref{Corollary: jensen upper bound} we have 
\begin{equation*}
    \mathbb{E}l(X) \geq l \left ( \Vert X \Vert _{p+1} \right ) \Leftrightarrow l^{-1}\left (\mathbb{E}l(X)   \right ) \geq \Vert X \Vert _{p+1}
\end{equation*}
for all $l \in \mathcal{L}_{p}$. Thus, $\inf _{l \in \mathcal{L}_{p}} l^{-1}\left (\mathbb{E}l(X)   \right ) \geq  \Vert X \Vert _{p+1}$. On the other hand, the function $m(x) = x^{p+1}$ satisfies $$m^{(2)}(x)x= p(p+1)x^{p-1}x =p (p+1) x^{p} =pm^{(1)}(x)$$ for all $x \geq 0$ and $m^{(k)} > 0$ 
for all $k=1,\ldots,p$, $x>0$. Thus, $m \in \mathcal{L}_{p}$. We have $m^{-1}\left (\mathbb{E}m(X)   \right ) =  \Vert X \Vert _{p+1}$. Hence, $\inf _{l \in \mathcal{L}_{p}} l^{-1}\left (\mathbb{E}l(X)   \right ) \leq  \Vert X \Vert _{p+1}$. We conclude that $\inf _{l \in \mathcal{L}_{p}} l^{-1}\left (\mathbb{E}l(X)   \right ) = \Vert X \Vert _{p+1}$ which proves the Theorem.     
\end{proof}

\begin{remark}
Theorems \ref{Theorem pratt} and \ref{Theorem: risk measures} can also be applied to concave and increasing utility functions using Proposition \ref{Prop: JEnsen}. 
\end{remark}

\subsection{Poisson approximation in the Wasserstein distance}

The Poisson approximation of a sum of binary random variables has received extension attention in the literature (for example, see \cite{barbour1992poisson}). In this section, we leverage Theorem \ref{Thm: Jensen} to derive a simple Poisson approximation result for the sum of independent binary random variables in terms of the Wasserstein distance. The Wasserstein distance has many desirable properties and has been recently popularized in statistics and machine learning. The Poisson approximation result presented in this section is simple to derive and does not use the Stein-Chen method that is typically employed to derive Poisson approximation results \citep{barbour2006stein}. Our result improves a similar  Poisson approximation result that is given in \cite{boutsikas2000bound}. 

The Wasserstein distance between two random variables $X$ and $Y$ is given by 
$$ d_{W} (\mathscr{L}(X),\mathscr{L}(Y)) = \inf \mathbb{E} | U - V| $$ 
where the infimum  is over all couplings $(U,V)$ of $ \mathscr{L}(X)$ and $\mathscr{L}(Y)$. 
We denote by $Po (x)$ the Poisson distribution with parameter $x$ and by $F(x)$ the cumulative distribution function of the exponential distribution with parameter $1$, i.e., $F(x) = 1 -\exp(-x)$.

\begin{theorem} \label{Theorem Poisson}
Suppose that $X_{1},\ldots,X_{n}$ are independent binary random variables on $\{0,1\}$ such that $\mathbb{P}(X_{i} = 1) = p_{i} $. Let $S_{n} = \sum _{i=1}^{n} X_{i}$.  Then  
\begin{equation}\label{Ineq: main}
        d_{W} \left (\mathscr{L} (S_{n}), Po\left(  \sum_{i=1}^{n} p_{i}  \right ) \right )  \leq   2n \left (\frac { \sqrt{ \sum _{i=1}^{n} p_{i}^{2} }} { \sqrt{n}}- F\left (\frac {\sqrt{ \sum _{i=1}^{n} p_{i}^{2} }} { \sqrt{n} }\right ) \right ).  
\end{equation}
\end{theorem}

\begin{remark}
Lemma 7 in \cite{boutsikas2000bound} shows that \begin{equation*}
        d_{W} \left (\mathscr{L} (S_{n}), Po\left(  \sum_{i=1}^{n} p_{i}  \right ) \right )  \leq     \sum_{i=1}^{n} p_{i}^{2}.
\end{equation*}
We note that inequality (\ref{Ineq: main}) is tighter than the inequality above. To see this note that  
$$ 2(x-(1-\exp(-x)) \leq  x^{2} $$
for every positive $x$. Letting $x= \sqrt{ \sum _{i=1}^{n} p_{i}^{2} / n}$ shows that 
$$  2n \left (\frac { \sqrt{ \sum _{i=1}^{n} p_{i}^{2} }} { \sqrt{n}}- F\left (\frac {\sqrt{ \sum _{i=1}^{n} p_{i}^{2} }} { \sqrt{n} }\right ) \right )  \leq \sum _{i=1}^{n} p_{i}^{2}. $$
\end{remark}

\begin{proof}
Define the function $f(x) := \exp(-x) - 1  + x - x^{2}/2$. Note that $f^{(1)}(0) = -\exp(-0) + 1 = 0$. It is straightforward to check that $f^{(1)}$, $f^{(2)}$, and $f^{(3)}$ are non-positive on $[0,1]$ so $-f$ satisfies the conditions of Theorem \ref{Thm: Jensen}. Consider the random variable $X$ that yields $p_{i}$ with probability $1/n$. Using Theorem \ref{Thm: Jensen} with the random variable $X$ and the function $f$ yields 
$$ n^{-1} \left ( \sum _{i=1}^{n} \left ( \exp(-p_{i}) - 1 +  p_{i} - \frac{ p_{i}^{2}} { 2} \right )  \right )  \leq \exp\left ( - \sqrt{  n^{-1} \sum _{i=1} ^{n} p_{i}^{2}  } \right ) - 1 +  \sqrt { n^{-1} \sum _{i=1} ^{n} p_{i}^{2} } - \frac{n^{-1} \sum _{i=1} ^{n} p_{i}^{2} } { 2}. $$
Rearranging the last inequality yields 
\begin{equation} \label{Eq. 1}
    2  \sum _{i=1}^{n} \left ( \exp(-p_{i}) - 1 +  p_{i} \right ) \leq - 2n \left ( 1 - \exp\left ( - \sqrt { n^{-1} \sum _{i=1} ^{n} p_{i}^{2} }   \right )  \right ) + 2 \sqrt { n \sum _{i=1} ^{n} p_{i}^{2} }
\end{equation}
Suppose that $Y_{1},\ldots,Y_{n}$ are independent random variables such that $\mathscr{L}(Y_{i}) = Po(p_{i})$. We have 
\begin{align*}
    d_{W} \left (\mathscr{L} (S_{n}), Po\left(  \sum_{i=1}^{n} p_{i}  \right ) \right ) & \leq  \sum _{i=1}^{n} d_{W}(\mathscr{L} (X_{i} ), Po( p_{i}) ) \\
    & = \sum _{i=1}^{n} \sum _{j=0}^{\infty} \vert \mathbb{P}(X_{i} \leq j) -  \mathbb{P}(Y_{i} \leq j)  \vert  \\
    & = \sum _{i=1}^{n} \sum _{j=0}^{\infty} \vert \mathbb{P}(X_{i} > j) -  \mathbb{P}(Y_{i} > j)  \vert  \\ 
    & = \sum _{i=1}^{n} \left (\vert p_{i} - (1- \exp(-p_{i}) ) \vert + \sum _{j=1}^{\infty}   \mathbb{P}(Y_{i} > j)   \right ) \\
     & = \sum _{i=1}^{n} \left (\vert p_{i} - (1- \exp(-p_{i}) )\vert + \mathbb{E}(Y_{i}) - \mathbb{P}(Y_{i} > 0)  \right ) \\
     & = \sum _{i=1}^{n} \left (\vert p_{i} - (1- \exp(-p_{i}) )\vert + p_{i}  -(1- \exp(-p_{i}) )  \right ) 
     \\
     & = 2 \sum _{i=1}^{n} \left (\exp(-p_{i}) - 1 +  p_{i}   \right ) \\
     & \leq - 2n \left ( 1 - \exp\left ( - \sqrt { n^{-1} \sum _{i=1} ^{n} p_{i}^{2} }   \right )  \right ) + 2 \sqrt { n \sum _{i=1} ^{n} p_{i}^{2} }  \\
     & = 2n \left (\frac { \sqrt{ \sum _{i=1}^{n} p_{i}^{2} }} { \sqrt{n}}- F\left (\frac { \sqrt{ \sum _{i=1}^{n} p_{i}^{2} }} { \sqrt{n} }\right ) \right ) 
\end{align*} 
which proves the Theorem. The first inequality follows from the subadditivity of the Wasserstein distance. The second inequality follows from inequality (\ref{Eq. 1}). 
\end{proof}

 \subsection{Lower and upper bounds on the moment generating function} \label{Sec: bounds on the MGF}

In this section we leverage  Theorem \ref{Thm: Jensen} to provide bounds on the moment generating function of a random variable.

For every integer $p$ the exponential function $\exp (x)$ satisfies the conditions of Lemma \ref{lemma: Taylor}, and hence, the Taylor remainder of the exponential function $\exp (x) - \sum _{j=0} ^{p} x^{j} / j!$ is a $(p,0,b)$-convex function for all $p \geq 2$.  We leverage this fact to provide lower and upper bounds on the moment generating function of a random variable.
We first provide a lower bound on the moment generating function of a random variable that is bounded from below. The bounds depend on the random variable's first $p$ moments.\footnote{\cite{zhang2018non} show that lower bounds on the moment generating function can be used to prove lower tail bounds for random variables. }

\begin{corollary} \label{Coro: exp}
 Let $X \in L^{p}$ be a random variable on $[0 ,\infty)$ where $p$ is a positive integer. For all $s \geq 0$ we have
\begin{align} \label{Ineq: MGF lower bound} 
   \mathbb{E} \exp (sX)  \geq  \exp \left (s \Vert X \Vert _{p} \right ) -  \sum _{j=0} ^{p-1} \frac{s^{j}\Vert X \Vert _{p} ^{j}}{j!} + \mathbb{E} \left ( \sum _{j=0} ^{p-1} \frac{s^{j}X^{j}} {j!} \right )
\end{align} 
\end{corollary}
\begin{proof}
Fix $s \geq 0$. From Lemma \ref{lemma: Taylor} the function  $\exp (x) - \sum _{j=0} ^{p-1} x^{j} / j!$ is $(p-1,0,\infty)$-convex. This implies that $g(x):=\exp (sx) - \sum _{j=0} ^{p-1} (sx)^{j} / j!$ is  $(p-1,0,\infty)$-convex. Applying Theorem \ref{Thm: Jensen} (see also Remark 1) for $g$ and rearranging yield inequality (\ref{Ineq: MGF lower bound}).
\end{proof}

Suppose that $X$ is a random variable on $[1,\infty)$. Then we can apply inequality (\ref{Ineq: MGF lower bound}) to the positive random variable $Y=\ln(X)$ and $s=1$ (assuming that $Y \in L^{p}$) to derive the following lower bound on the expected value of $X$: 
\begin{align} \label{Ineq: Lower bound on expected value} 
   \mathbb{E} X  \geq  \exp \left (\Vert \ln(X) \Vert _{p} \right ) -  \sum _{j=0} ^{p-1} \frac{\Vert \ln(X) \Vert _{p} ^{j}}{j!} + \mathbb{E} \left ( \sum _{j=0} ^{p-1} \frac{\ln(X)^{j}} {j!} \right )
\end{align} 
When $X$ is the random variable that yields $x_{i} > 0$, $i=1,\ldots,n$ with probability $1/n$ then inequality (\ref{Ineq: Lower bound on expected value}) generalizes the well-known AM-GM inequality which corresponds to $p=1$.

We now provide an upper bound on the moment generating function. 
\begin{corollary} \label{Coro: exp upper bound}
 Let $X$ be a random variable on $[0 ,b]$ for some $b>0$ where $p$ is a positive integer. For all $s \geq 0$ we have
\begin{align} \label{Ineq: MGF bound}
\mathbb{E} \exp(sX)  \leq \frac{\mathbb{E} X^{p} }{b^{p}}\left (\exp(sb) - \sum _{j=0} ^{p-1} \frac{s^{j}b^{j}}{j!}  \right )+ \mathbb{E} \left ( \sum _{j=0} ^{p-1} \frac{s^{j}X^{j}} {j!} \right ).
\end{align}  
 \end{corollary}

\begin{proof}
Fix $s \geq 0$. From Lemma \ref{lemma: Taylor} the function $g(x):=\exp (sx) - \sum _{j=0} ^{p-1} (sx)^{j} / j!$ is $(p-1,0,b)$-convex. Applying Corollary \ref{Corollary: jensen upper bound} to $g$ and noting that $g(0)=0$ prove the Corollary.  
\end{proof}

Corollary \ref{Coro: exp upper bound} is also proved in \cite{light2020concentration} using a different approach and is fundamental in deriving Hoeffding type concentration inequalities. 

\subsection{Lower bounds on the log-likelihood function} 

In the presence of hidden variables, lower bounds on the log-likelihood function are important in computing the maximum likelihood estimator. For example, the popular  expectation maximization algorithm \citep{dempster1977maximum} computes the maximum likelihood parameters using a lower bound on the  log-likelihood function. 

Consider the following standard estimation problem. We have a training set $x_{i} \in X$, latent (hidden) variables $z_{i}$, some parameters $\theta \in \Theta$, and a likelihood function $p(x,z|\theta)$ where $i=1,\ldots,n$. We assume for simplicity that there is a finite number of latent variables. The maximum likelihood estimate is determined by maximizing  
$$ l(\theta):=\sum _{i=1}^{n} \ln \left ( \sum_{z_{i}} p(x_{i},z_{i}|\theta) \right ) $$
with respect to $\theta$. In many practical cases, this optimization problem is not tractable and a lower bound for the log-likelihood function $l$ is essential for computing the  maximum likelihood estimate. The popular expectation maximization algorithm uses the following lower bound derived from Jensen's inequality:  
$$\sum _{i=1}^{n} \ln \left ( \sum_{z_{i}} p(x_{i},z_{i}|\theta) \right ) \geq \sum _{i=1}^{n} \sum_{z_{i}} q_{i} (z_{i}) \ln \left ( \frac {  p(x_{i},z_{i}|\theta) } {q_{i}(z_{i}) } \right )  $$
for any probability mass functions $q_{1},\ldots,q_{n}$ on $Z$, $q_{i} (z_{i}) > 0$ for all $z_{i}$, and $\sum_ {z_{i} } q_{i}(z_{i}) = 1$. 
Using Theorem \ref{Thm: Jensen} we provide a tighter bound for the log-likelihood function.

\begin{theorem}
For any probability mass functions $q_{1},\ldots,q_{n}$ on $Z$, $q_{i} (z_{i}) > 0$ for all $z_{i}$, and $\sum_ {z_{i} } q_{i}(z_{i}) = 1$ we have 
\begin{equation} \label{ineq: Likelihood}
   \sum _{i=1}^{n} \ln \left ( \sum_{z_{i}} p(x_{i},z_{i}|\theta) \right ) \geq \sum _{i=1}^{n} \left (\ln \left (b_{i} - \Vert b_{i} - X_{i} \Vert _{2} \right ) - \frac { b_{i} - \Vert b_{i} - X_{i} \Vert _{2} - \mathbb{E}(X_{i}) } { b_{i}}   \right ) \geq \sum _{i=1}^{n} \mathbb{E}\ln (X_{i}) 
\end{equation} 
where $X_{i}$ is the random variable that assigns the value $p(x_{i},z_{i}|\theta)/q_{i}(z_{i})$ with probability $q_{i}(z_{i})$ and $b_{i} = \max _{z_{i}} p(x_{i},z_{i}|\theta)/q_{i}(z_{i})$ for all $i=1,\ldots,n$. 
\end{theorem}

\begin{proof}
Consider the function $f(x):=\ln(x) - x/b$ on $(0,b]$, $b>0$. Then $f^{(1)}(b) = 0$, $f^{(1)} \geq 0$, $f^{(2)} \leq 0$, and $f^{(3)} \geq 0$. Thus, we can apply Proposition \ref{Prop: JEnsen} to conclude that $\mathbb{E}f(X) \leq  f (b - \Vert b - X \Vert _{2})$, i.e., $$\mathbb{E}\ln(X) - \frac { \mathbb{E}(X) } {b} \leq \ln (b - \Vert b - X \Vert _{2}) - \frac{(b - \Vert b - X \Vert _{2})}{b} $$  
for every random variable $X$ on $(0,b]$. Applying the last inequality for the random variables $X_{i}$'s defined in the statement of the Theorem and summing over $i$ yields the right-hand-side of inequality (\ref{ineq: Likelihood}). 

Using the monotonicity of the p-norm we have $b - \Vert b - X \Vert _{2} \leq b - \Vert b - X \Vert _{1} = \mathbb{E}X$ for every random variable $X$ on $(0,b]$. In addition, because $\ln ( \cdot )$ is a concave function, for every numbers $d,c$ such that $b \geq d > c>0$ we have $(\ln(d) - \ln (c) ) /(d-c) \geq 1/d \geq 1/b$. Thus, 
$b ( \ln(d) - \ln (c) ) \geq d-c$. Applying the last inequality for $d =\mathbb{E}X $ and $c= b - \Vert b - X \Vert _{2} $  yields 
$$ b( \ln (\mathbb{E}X) -  \ln ( b - \Vert b - X \Vert _{2}) ) \geq \mathbb{E}X - (b - \Vert b - X \Vert _{2})$$ 
for every random variable $X$ on $(0,b]$. 
 We have
\begin{align*}
     \sum _{i=1}^{n} \ln \left ( \sum_{z_{i}} p(x_{i},z_{i}|\theta) \right ) & =  \sum _{i=1}^{n} \ln \left (\sum_{z_{i}}  q_{i}(z_{i}) \frac {  p(x_{i},z_{i}|\theta) } {q_{i}(z_{i}) } \right ) \\
     & =\sum _{i=1}^{n} \ln \left ( \mathbb{E}X_{i} \right ) \\
     & \geq \sum _{i=1}^{n} \left (\ln \left (b_{i} - \Vert b_{i} - X_{i} \Vert _{2} \right ) - \frac { b_{i} - \Vert b_{i} - X_{i} \Vert _{2} - \mathbb{E}(X_{i}) } { b_{i}}   \right )
\end{align*}
which proves the left-hand-side of inequality  (\ref{ineq: Likelihood}).  
\end{proof}

\subsection{Hermite-Hadamard inequalities} \label{Sec: HH ineq}

Hermite-Hadamard type inequalities have numerous applications in various fields of mathematics (see \cite{dragomir2003selected}). 
The classical Hermite-Hadamard inequality states that for a convex function $f :\left [a ,b\right ] \rightarrow \mathbb{R}$ we have
\begin{equation}f\genfrac{(}{)}{}{}{a +b}{2} \leq \frac{1}{b -a}\int _{a}^{b}f(x)dx \leq \frac{f\left (a\right ) +f(b)}{2} . \label{Ineq: HH}
\end{equation}             

Generalizations and refinements of the Hermite-Hadamard  inequality have received a significant attention recently.\footnote{
See for example \cite{niculescu2009hermite}, \cite{de2009general},   \cite{sarikaya2013hermite}, \cite{mako2017approximate},   \cite{chen2017hermite}, and \cite{olbrys2019problem}.}

We now provide we provide a generalization of Hermite-Hadamard inequality for $(p,a,b)$-convex  functions. 


\begin{theorem} \label{Theorem: NEW HH}
\label{Thm: HH}Fix an integer $p \geq 1$ and $a<b$. Let $f$ be a $(p-1,a,b)$-convex function. Then   
 \begin{equation}f\left (\frac{1}{(p +1)^{1/p}}b +\left (1 -\frac{1}{(p +1)^{1/p}}\right )a\right ) \leq \frac{1}{b -a}\int _{a}^{b}f(x)dx \leq \frac{p}{p +1}f(a) +\frac{1}{p +1}f(b) . \label{Ineq: NEW HH}
\end{equation}
\end{theorem}

\begin{proof}
Let $f$ be a $(p-1,a,b)$-convex function for some $p \geq 1$ and $a <b$.   

Let $X$ be the continuous uniform random variable on $[a ,b]$, i.e., $\mathbb{P}(X \in [a,d] ) = (d-a)/(b-a)$ for all $ d \in [a,b]$. Then Theorem \ref{Thm: Jensen} implies that 
\begin{align*}\frac{1}{b -a}\int _{a}^{b}f(x)dx &  \geq f\left (a +\left (\int _{a}^{b}\left (x -a\right )^{p}(b -a)^{ -1}dx\right )^{1/p}\right ) \\
 & = f\left (a +\genfrac{(}{)}{}{}{(b -a)^{p +1}}{(p +1)(b -a)}^{1/p}\right ) \\
 & = f\left (\frac{1}{(p +1)^{1/p}}b +\left (1 -\frac{1}{(p +1)^{1/p}}\right )a\right )\end{align*}
 which proves the left-hand-side of inequality (\ref{Ineq: NEW HH}). 

Consider the random variable that yields $a$ with probability $t \in (0,1)$ and $b$ with probability $1- t$. From Theorem \ref{Thm: Jensen} for all $f \in \mathfrak{I}(p-1,a,b)$ we have 
\begin{equation*}
    (1-t)f(a) + tf(b) \geq f\left (a+(t(b-a)^{p})^{1/p} \right ) = f(a+t^{1/p}(b-a)). 
\end{equation*}  

Let $t= \lambda^{1/p}$. We conclude that
\begin{align*} f(\lambda b +(1 -\lambda )a) \leq \lambda ^{p}f(b) +\left (1 -\lambda ^{p}\right )f(a)\end{align*}
for all $0 \leq \lambda  \leq 1$.
Integrating both sides of the last inequality implies
\begin{align*}& \int _{0}^{1}f\left (\lambda a +\left (1 -\lambda \right )b\right )d\lambda    \leq f(b)\int _{0}^{1}\lambda ^{p}d\lambda  +f\left (a\right )\int _{0}^{1}\left (1 -\lambda ^{p}\right )d\lambda  \\
 &  \Leftrightarrow \frac{1}{b -a}\int _{a}^{b}f(x)dx \leq \frac{1}{p +1}f(b) +\frac{p}{p +1}f(a)\end{align*}
 which proves the right-hand-side of inequality (\ref{Ineq: NEW HH}). 
\end{proof}

\begin{remark}

For all $p \geq 1$, inequality (\ref{Ineq: NEW HH}) provides a tighter lower and an upper bound on  $\frac{1}{b -a}\int _{a}^{b}f(x)dx$ than inequality (\ref{Ineq: HH}) because 
\begin{equation*}\frac{f(a) +f(b)}{2} \geq \frac{p}{p +1}f(a) +\frac{1}{p +1}f(b)
\end{equation*} and 
\begin{equation*}\frac{1}{(p +1)^{1/p}}b +\left (1 -\frac{1}{(p +1)^{1/p}}\right )a \geq \frac{a +b}{2} .
\end{equation*}
\end{remark}


Theorem \ref{Ineq: Jensen} can be used for deriving inequalities for the Taylor reminder of some functions of interest. 
From Lemma \ref{lemma: Taylor} the function $T_{p}(x) : = \exp(x) -  \sum _{j=0} ^{p} x^{j} / j!$ is $(p,0,b)$-convex. Applying Theorem \ref{Theorem: NEW HH} for the $(p-1,0,b)$-convex function $T_{p-1}$, and noting that $T_{p-1}(0)=0$ and $\int _{0}^{b}T_{p-1}(x)dx = T_{p}(b)$ yield  the bounds:
 \begin{equation}T_{p-1} \left (\frac{1}{(p +1)^{1/p}}b \right ) \leq \frac{T_{p}(b)}{b} \leq  \frac{1}{p +1}T_{p-1}(b) 
 \end{equation}
for all $b>0$ and every integer $p \geq 1$.

Many Hermite-Hadamard type inequalities that hold for convex functions can be generalized for $(p,a,b)$-convex functions. We provide two examples. First, we  generalize an inequality for differentiable convex mappings that was proved in \cite{dragomir1998two} (see Theorem \ref{Thm: differe}). Second, in Section \ref{Fractional} we provide inequalities for fractional integrals (see Theorem \ref{Thm: HH Fractional}).

\begin{theorem} \label{Thm: differe}
 Suppose that $f :[a ,b] \rightarrow \mathbb{R}$ is differentiable on $[a ,b]$. Suppose that $\vert f^{ \prime }\vert $ is a $(p-1,a,b)$-convex  function. Then the following inequality holds: 
 $$ \left \vert \frac{f(a) +f(b)}{2} -\frac{1}{(b -a)}\int _{a}^{b}f(x)dx\right \vert \leq \frac{(b -a)}{4} \left [ \frac{2(p +0.5^{p})}{(p +1)(p +2)}\vert f^{ \prime }(a)\vert  +\left (1 -\frac{2(p +0.5^{p})}{(p +1)(p +2)}\right )\left \vert f^{ \prime}(b)\right \vert \right].$$
\end{theorem}

\begin{proof}
We have 
\begin{align*}\left \vert \frac{f(a) +f(b)}{2} -\frac{1}{(b -a)}\int _{a}^{b}f(x)dx\right \vert & =\left \vert \frac{b -a}{2}\int _{0}^{1}(1 -2t)f^{ \prime }(ta +(1 -t)b)dt\right \vert  \\
&\leq \frac{b -a}{2}\int _{0}^{1}\vert 1 -2t\vert \left \vert f^{ \prime }(ta +(1-t)b )\right \vert dt \\
 &\leq \frac{b -a}{2}\int _{0}^{1}\vert 2t - 1 \vert \left [t^{p}\vert f^{\prime }(a)\vert  +(1 -t^{p})\vert f^{ \prime }(b)\vert \right ]dt \\
 &=\frac{(b -a)}{4} \left [ \frac{2(p +0.5^{p})}{(p +1)(p +2)}\vert f^{ \prime }(a)\vert  +\left (1 -\frac{2(p +0.5^{p})}{(p +1)(p +2)}\right )\left \vert f^{ \prime}(b)\right \vert \right].\end{align*}
The first equality follows from Lemma 2.1 in \cite{dragomir1998two}. The second inequality follows because $f$ is a $(p-1,a,b)$-convex  function (see  the proof of Theorem \ref{Thm: HH}). The last equality follows from noting that 
$$\int _{0}^{0.5}(1 -2t)t^{p} =\frac{0.5^{p +1}}{(p +1)(p +2)} \text { and } \int _{0.5}^{1}(2t -1)t^{p} =\frac{p +0.5^{p +1}}{(p +1)(p +2)}.$$
Thus,  
$$\int _{0}^{1}\vert 2t -1\vert t^{p} =\frac{p +0.5^{p}}{(p +1)(p +2)} \text{ and  } \int _{0}^{1}\vert 2t -1\vert (1-t^{p}) = 0.5 - \frac{p +0.5^{p}}{(p +1)(p +2)}.$$
\end{proof}

\subsection{Inequalities for fractional integrals}\label{Fractional}

We now state and prove Hermite-Hadamard inequalities that involve fractional integrals and $(p,a,b)$-convex functions.

\begin{definition}
Let $f :[a ,b] \rightarrow \mathbb{R}$, $0 \leq a<b$. The Riemann--Liouville integrals $I_{a +}^{\alpha }f$ and $I_{b -}^{\alpha }f$ of order $\alpha  >0$ are defined by \begin{equation*}I_{a +}^{\alpha }f(x) : =\frac{1}{\Gamma (\alpha )}\int _{a}^{x}(x -t)^{\alpha  -1}f(t)dt ,
\end{equation*} 
for $x >a$ and \begin{equation*}I_{b -}^{\alpha }f(x) : =\frac{1}{\Gamma (\alpha )}\int _{x}^{b}(t -x)^{\alpha  -1}f(t)dt,
\end{equation*}
for $x <b$ where $\Gamma (\alpha )$ is the gamma function and $I_{a +}^{0}f(x) =I_{b-}^{0}f(x) =f(x)$. 
\end{definition}

\cite{sarikaya2013hermite} prove the following Hermite-Hadamard inequality for fractional integrals: Let $f:[a,b] \rightarrow \mathbb{R}$, $0 \leq a < b$ be a positive and convex function. We have
\begin{equation}f\genfrac{(}{)}{}{}{a +b}{2} \leq \frac{\Gamma (\alpha  +1)}{2(b -a)^{\alpha }}\left (I_{a +}^{\alpha }f(b) +I_{b -}^{\alpha }f(a)\right ) \leq \frac{f(a) +f(b)}{2} . \label{Ineq: HH-fractional}
\end{equation}
We now generalize the last inequality for the class of $(p,a,b)$-convex  functions.

For $p \geq 1$ and $\alpha \geq 0$, define \begin{equation}\gamma (p,\alpha) : =\frac{\alpha }{2(\alpha  +p)} +\frac{\Gamma (\alpha  +1)\Gamma (p +1)}{2\Gamma (\alpha  +p +1)} .
\end{equation}

\begin{theorem}  \label{Thm: HH Fractional}
Let $p \geq 1$ be an integer and let $f$ be a $(p-1,a,b)$-convex function where $0 \leq a < b$. Then for every $\alpha > 0 $ we have
\begin{equation} \label{Ineq: Fractional NEW}
f\left (\gamma (p,\alpha)^{1/p}b +(1 -\gamma (p,\alpha)^{1/p})a\right ) \leq \frac{\Gamma (\alpha  +1)}{2(b -a)^{\alpha }}\left (I_{a +}^{\alpha }f(b) +I_{b -}^{\alpha }f(a)\right ) \leq \gamma (p,\alpha)f(b) +(1 -\gamma (p,\alpha))f(a) .
\end{equation}
\end{theorem}

\begin{proof}
Let $p \geq 1$ be an integer and let $\alpha > 0$. Let $X$ be a random variable on $[a ,b]$ whose probability density function $g$ is given by \begin{equation*}g(x) =\frac{\alpha }{2(b -a)^{\alpha }}\left (\left (x -a\right )^{\alpha  -1} +\left (b -x\right )^{\alpha  -1}\right ) \text{ on } [a,b].
\end{equation*}Note that $\int _{a}^{b}g(x) dx =1$ and $g(x) \geq 0$ for all $x \in [a ,b]$ so $g$ is a density function. 

We have 
\begin{equation*}\mathbb{E}f(X) =\frac{\alpha }{2(b -a)^{\alpha }}\int _{a}^{b}f(x)\left (\left (x -a\right )^{\alpha  -1} +\left (b -x\right )^{\alpha  -1}\right )dx =\frac{\Gamma (\alpha  +1)}{2(b -a)^{\alpha }}\left (I_{b -}^{\alpha }f(a) +I_{a +}^{\alpha }f(b)\right )
\end{equation*}
where we use the fact that $\alpha \Gamma (\alpha ) =\Gamma (\alpha  +1)$.

We also have
\begin{align*}f\left (a +\left (\mathbb{E}\left (X -a\right )^{p}\right )^{1/p}\right ) &  =f\left (a +\left (\frac{\alpha }{2(b -a)^{\alpha }}\int _{a}^{b}\left (x -a\right )^{p}\left (\left (x -a\right )^{\alpha  -1} +\left (b -x\right )^{\alpha  -1}\right )dx\right )^{1/p}\right ) \\
 &  =f\left (a +\left (\frac{\alpha }{2(b -a)^{\alpha }}\left (\frac{(b -a)^{\alpha  +p}}{\alpha  +p} +\frac{\Gamma (\alpha )\Gamma (p +1)(b -a)^{\alpha  +p}}{\Gamma (\alpha  +p +1)}\right )\right )^{1/p}\right ) \\
 &  =f\left (a +\left (\frac{(b -a)^{p}}{2}\left (\frac{\alpha }{\alpha  +p} +\frac{\Gamma (\alpha  +1)\Gamma (p +1)}{\Gamma (\alpha  +p +1)}\right )\right )^{1/p}\right ) \\
 &  =f\left (\gamma (p,\alpha)^{1/p}b +(1 -\gamma (p,\alpha)^{1/p})a\right )\end{align*}The second equality follows from\ \begin{equation*}\int _{a}^{b}(x -a)^{p}(x -a)^{\alpha  -1}dx =(b -a)^{\alpha  +p}/(\alpha  +p)
\end{equation*}\ and  \begin{equation*}\int _{a}^{b}(x -a)^{p}(b -x)^{\alpha  -1}dx =\frac{\Gamma (\alpha )\Gamma (p +1)(b -a)^{p +\alpha }}{\Gamma (p +\alpha  +1)} .
\end{equation*}
Hence, for a $(p-1,a,b)$-convex function $f$ we can use Theorem \ref{Thm: Jensen} to conclude that 
\begin{equation*}f\left (\gamma (p,\alpha)^{1/p}b +(1 -\gamma (p,\alpha)^{1/p})a\right ) \leq \frac{\Gamma (\alpha  +1)}{2(b -a)^{\alpha }}\left (I_{a +}^{\alpha }f(b) +I_{b -}^{\alpha }f(a)\right )
\end{equation*}
which proves the left-hand-side of inequality (\ref{Ineq: Fractional NEW}). 

Let $f$ be a $(p-1,a,b)$-convex function. 
From the proof of Theorem \ref{Thm: Jensen} we have 
\begin{equation*}f(\lambda b +(1 -\lambda )a) \leq \lambda ^{p}f(b) +\left (1 -\lambda ^{p}\right )f(a)
\end{equation*}and \begin{equation*}f((1 -\lambda )b +\lambda a) \leq (1 -\lambda )^{p}f(b) +\left (1 -(1 -\lambda )^{p}\right )f(a)
\end{equation*}

for all $0 \leq \lambda  \leq 1$. 
Adding the last two inequalities yields
\begin{equation*}f(\lambda b +(1 -\lambda )a) +f((1 -\lambda )b +\lambda a) \leq f(b)(\lambda ^{p} +(1 -\lambda )^{p}) +f(a)(2 -\lambda ^{p} -(1 -\lambda )^{p}) .
\end{equation*} Multiplying each side of the last inequality by $\lambda ^{\alpha  -1}$ and integrating with respect to $\lambda $ over $[0,1]$ yield 
\begin{align} \label{Ineq: PROOF FRACTIONAL}
    \begin{split}
         & \int_{0}^{1}  \lambda^{\alpha -1}  f(\lambda b +(1 -\lambda )a) d\lambda +\int_{0}^{1} \lambda^{\alpha -1} f((1 -\lambda )b +\lambda a)  d\lambda\\ 
     & \leq  f(b) \int_{0}^{1} \lambda^{\alpha -1} (\lambda ^{p} +(1 -\lambda )^{p})  d\lambda +  f(a)\int_{0}^{1}  \lambda^{\alpha -1} (2 -\lambda ^{p} -(1 -\lambda )^{p})  d\lambda.
     \end{split}
\end{align}
Note that 
$$  \int_{0}^{1}  \lambda^{\alpha -1}  f(\lambda b +(1 -\lambda )a) d\lambda =  \int_{a}^{b}  \left (\frac{x-a}{b-a} \right )^{\alpha -1}  \frac{f(x) }{b-a} dx =\frac{1}{(b-a)^{\alpha}}\Gamma (\alpha ) I_{b -}^{\alpha } f(a).$$ 
Similarly 
$$ \int_{0}^{1} \lambda^{\alpha -1} f((1 -\lambda )b +\lambda a)  d\lambda= \frac{1}{(b-a)^{\alpha}}\Gamma (\alpha ) I_{a +}^{\alpha }f(b). $$
Using inequality (\ref{Ineq: PROOF FRACTIONAL}) yields
\begin{equation*}\frac{\Gamma (\alpha )}{(b -a)^{\alpha }}\left (\int _{a +}^{\alpha }f(b) +\int _{b -}^{\alpha }f(a)\right ) \leq f(b)\left (\frac{1}{\alpha  +p} +\frac{\Gamma (\alpha )\Gamma (p +1)}{\Gamma (\alpha  +p +1)}\right ) +f(a)\left (\frac{2}{\alpha } -\frac{1}{\alpha  +p} -\frac{\Gamma (\alpha )\Gamma (p +1)}{\Gamma (\alpha  +p +1)}\right ).
\end{equation*}Multiplying each side of the last inequality by $\alpha /2$ proves the right-hand-side of inequality (\ref{Ineq: Fractional NEW}). 
\end{proof}

For all $\alpha > 0$ note that 
\begin{equation*}\gamma (1,\alpha) =\frac{\alpha }{2(\alpha  +1)} +\frac{\Gamma (\alpha  +1)\Gamma (2)}{2\Gamma (\alpha  +2)} =\frac{1}{2}\left (\frac{\alpha }{\alpha  +1} +\frac{1}{\alpha  +1}\right ) =\frac{1}{2} .
\end{equation*}Thus, Theorem \ref{Thm: HH Fractional} reduces to inequality (\ref{Ineq: HH-fractional}) for $p =1$.

\section{Summary}
This paper studies inequalities for functions that are ``very" convex.  These inequalities are simple and easy to apply. We demonstrate the usefulness of these inequalities in a variety of applications from different fields. We foresee additional beneficial applications of our results for studying settings that involve convex functions.

\section{Appendix}

\begin{proof}[Proof of Theorem \ref{Thm: Jensen}]
(i) Let $p \geq 1$ be an integer and $a <b$. Suppose that $f$ is a $(p ,a ,b)$-convex function. We can assume that $f^{(p)}$ is  differentiable.\footnote{Because  $f^{(p)}$ is convex and increasing there exists a sequence of continuously differentiable functions $f_{n}^{(p)}$ such that $ \lim _{n \rightarrow \infty} f_{n}^{(p)} = f^{(p)}$ (see \cite{light2019family}) and the proof follows from an application of the dominated convergence theorem. } 

Define the function $k_{p}(y) =f(a +y^{1/(p+1)})$ on $[0 ,(b -a)^{p+1}]$.  We first show that $k_{p}$ is a convex function on $[0 ,(b -a)^{p+1}]$. 

 $k_{p}$ is convex on $[0 ,(b -a)^{p+1}]$ if and only if $k_{p}^{(2)}(y) \geq 0$ for all $y$ in $(0 ,(b -a)^{p+1})$, i.e., 
 \begin{equation*}f^{(2)}(a +y^{1/(p+1)})y^{\frac{2 - 2p}{p+1}} -pf^{(1)}(a +y^{1/(p+1)})y^{\frac{1 -2p}{p+1}} \geq 0.
\end{equation*}
Defining $x =a +y^{1/(p+1)}$ and rearranging yields 
\begin{equation}f^{(2)}(x)(x -a) -p f^{(1)}(x) \geq 0. \label{Jensen: ineq1}
\end{equation}
Thus, $k_{p}$ is convex on $[0 ,(b -a)^{p+1}]$ if and only if  inequality (\ref{Jensen: ineq1}) holds for all $x \in [a ,b]$.

Because $f^{(p)}$ is convex the derivative of $f^{(p)}$ is increasing. Using the fact that $f^{(p)}(a) =0$ we have
\begin{equation}f^{(p)}(x) =f^{(p)}(a) +\int _{a}^{x}f^{(p+1)}(t)dt =\int _{a}^{x}f^{(p+1)}(t)dt \leq \int _{a}^{x}f^{(p+1)}(x)dt =(x -a)f^{(p+1)}(x). \label{Jensen: ineq2}
\end{equation}
 Thus, if $p = 1$ then inequality (\ref{Jensen: ineq1}) holds and $k_{1}$ is convex on  $[0 ,(b -a)^{2}]$. 

To prove that inequality  (\ref{Jensen: ineq1}) holds for $p \geq 2$, define $z(x) =f^{(2)}(x)(x -a) -pf^{(1)}(x)$ on $[a ,b]$. Note that $z^{(1)}(x) =f^{(3)}(x)(x -a) -f^{(2)}(x)(p -1)$, and more generally \begin{equation*}z^{(k)}(x) =f^{(k +2)}(x)(x -a) -f^{(k +1)}(x)(p -k)
\end{equation*}
for all $k=1,\ldots,p-1$. 

Because $f^{(k)}(a) =0$ for all $k=1,\ldots,p$, we have $z^{(k)}(a) =0$ for all  $k=0,\ldots,p-1$.  

Inequality (\ref{Jensen: ineq2}) yields $z^{(p - 1)}(x) \geq 0$ for all $x \in [a ,b]$. Thus, $z^{(p -2)}$ is an increasing function. Combining this with the fact that $z^{(p -2)}(a) =0$ implies that $z^{(p -2)}(x) \geq 0$ for all $x \in [a ,b]$. Using the same argument as above it follows by induction that $z^{(j)}(x) \geq 0$ for all $x \in [a ,b]$ and all $j=0,\ldots,p-1$. In particular, $z^{(0)}(x) :=z(x) \geq 0$ for all $x \in [a ,b]$.  We conclude that inequality  (\ref{Jensen: ineq1}) holds for all $p \geq 2$, i.e., the function $k_{p}$ is convex on $[0 ,(b -a)^{p+1}]$.   


 Let $X$ be a random variable on $[a ,b]$. From Jensen's inequality (applied to the random variable $Y=(X-a)^{p+1}$ on $[0 ,(b -a)^{p+1}]$)  we have $$\mathbb{E}k_{p}((X -a)^{p+1}) \geq k_{p}(\mathbb{E}(X -a)^{p-1}).$$ That is, 
\begin{equation*}\mathbb{E}f(X) \geq f\left (a +\left (\mathbb{E}(X -a)^{p+1}\right )^{1/(p+1)}\right ) 
\end{equation*}
which proves part (i).

(ii) From part (i) the function $k_{p}(y) =f(a +y^{1/(p+1)})$ is convex on $[0 ,(b -a)^{p+1}]$. 

Because $k_{p}$ is convex we have $k_{p}(0) \geq k_{p}(y) +k_{p}^{(1)} (y) \left (0-y\right )$ for all $y \in [0 ,(b -a)^{p+1}]$. Thus,
\begin{equation*} f(a) \geq f(a+y^{1/(p+1)}) - f^{(1)}(a+y^{1/(p+1)})\frac{y^{1/(p+1)}}{p+1}. 
\end{equation*}
Using the fact that $f(a)=0$ and defining $x=a+y^{1/(p+1)}$ yield $(p+1)f(x) \leq f^{(1)}(x)(x-a)$.

Note that 
\begin{equation*}
    g^{(1)}(x) = \frac{f^{(1)}(x)(x-a)^{p+1} - (p+1)(x-a)^{p}f(x)}{((x-a)^{p+1})^{2}} \geq 0 \end{equation*}
if and only if $$f^{(1)}(x)(x-a) \geq (p+1)f(x).$$ We conclude that $g$ is increasing. 
\end{proof}

\begin{proof}[Proof of Corollary \ref{Corollary: jensen upper bound}]
Consider the random variable that yields $a$ with probability $t \in (0,1)$ and $b$ with probability $1- t$. From Theorem \ref{Thm: Jensen} for all $f \in \mathfrak{I}(p,a,b)$ we have 
\begin{equation*}
    (1-t)f(a) + tf(b) \geq f\left (a+(t(b-a)^{p+1})^{1/(p+1)} \right ) = f(a+t^{1/(p+1)}(b-a)). 
\end{equation*}  
For $\left ( (x-a)/(b-a) \right )^{p+1}  = t$ where $x \in [a ,b]$ we have 
\begin{equation*}
   \left(1-\frac{(x-a) ^{p+1}}{(b-a)^{p+1}} \right)f(a) + \frac{(x-a) ^{p+1}}{(b-a)^{p+1}} f(b) \geq f(x) 
\end{equation*} 
Let $X$ be a random variable on $[a,b]$. Taking expectations in both sides of the last inequality yields 
\begin{equation*}
   \left(1-\frac{\mathbb{E} (X-a) ^{p+1}}{(b-a)^{p+1}} \right)f(a) + \frac{\mathbb{E} (X-a) ^{p+1}}{(b-a)^{p+1}} f(b) \geq \mathbb{E} f(X)
\end{equation*}  
which proves the Corollary. 
\end{proof}

\bibliographystyle{ecta}
\bibliography{HH}

\end{document}